\documentclass[12pt]{article}%
\usepackage[colorlinks,bookmarks=true]{hyperref}
\usepackage{amsmath}
\usepackage{graphicx}
\usepackage{amsfonts}
\usepackage{amssymb}%
\providecommand{\U}[1]{\protect\rule{.1in}{.1in}}
\newtheorem{theorem}{Theorem} [section]

\newtheorem{lemma}{Lemma}[section]

\newtheorem{remark}{Remark} [section]

\newenvironment{proof}[1][Proof]{\textbf{#1.} }{\ \rule{1em}{1em}}
\numberwithin{equation}{section}

\begin{document}

\title{Linear instability of Vlasov-Maxwell systems revisited-A Hamiltonian approach}
\author{Zhiwu Lin\\School of Mathematics\\Georgia Institute of Technology\\Atlanta, GA 30332, USA}
\date{}
\maketitle

\begin{abstract}
We consider linear stability of steady states of 1$\frac{1}{2}$ and 3D
Vlasov-Maxwell systems for collisionless plasmas. The linearized systems can
be written as separable Hamiltonian systems with constraints. By using a
general theory for separable Hamiltonian systems, we recover the sharp linear
stability criteria obtained previously by different approaches. Moreover, we
obtain the exponential trichotomy estimates for the linearized Vlasov-Maxwell
systems in both relativistic and nonrelativistic cases.

\end{abstract}

\section{Introduction}

Consider a plasma at high temperature, of low density such that collisions can
be ignored compared with the electromagnetic forces. Such a collisionless
plasma is modeled by the relativistic Vlasov-Maxwell system. In applications,
the classical Vlasov-Maxwell system is also considered when the effect of
special relativity is negligible. One of the central problems in the theory of
plasmas is to understand plasma stability and instability. The stability
problem of Vlasov plasmas is complicated partly because of the instability is
usually due to the collective behavior of all the particles. This makes the
instability problem highly nonlocal and difficult to study analytically. It is
also challenging numerically since the distribution is defined in the phase
space with a dimension doubling the space dimension. In a series of works
(\cite{lin-strauss1} \cite{lin-strauss2} \cite{lin-strauss-nl}), a sharp
stability criterion was obtained for certain equilibria of $1\frac{1}{2}$D
Vlasov-Maxwell system and $3\,$D relativistic Vlasov-Maxwell system with
cylindrical symmetry. More specifically, when the steady distribution function
has a monotonic dependence on the particle energy, the number of unstable
modes of linearized RVM systems is shown to be equal to $n^{-}\left(
\mathcal{L}^{0}\right)  $, the number of negative eigenvalues of a
self-adjoint operator $\mathcal{L}^{0}\ $(see (\ref{defn-cal-L0-1.5}) and
(\ref{defn-cal-L0-3d})) acting on functions depending only on space variables.
In these works, the existence of unstable eigenfunctions was shown by
introducing a family of non-local self-adjoint operators $\mathcal{A}%
^{\lambda\ }$for electromagnetic potentials, where the positive parameter
$\lambda\ $is the possible unstable eigenvalue. Then an instability criteria
was obtained by using a continuity argument to exploit the gap of numbers of
negative eigenvalues of $\mathcal{A}^{\lambda\ }$when $\lambda\rightarrow
\infty$ and $\lambda\rightarrow0+$. The proof was particularly involved for
the 3D Vlasov-Maxwell systems (\cite{lin-strauss2}) since the self-adjoint
formulation of $\mathcal{A}^{\lambda\ }$relied on a careful choice of the
gauge condition of the electromagnetic potentials. Moreover, the operator
$\mathcal{A}^{\lambda\ }$in $3$D has an infinite number of negative
eigenvalues and a truncation of $\mathcal{A}^{\lambda\ }$ has to be introduced
in order to use the continuity argument. The linear stability criterion
$\mathcal{L}^{0}\geq0$ was proved by studying the invariant functionals of the
linearized Vlasov-Maxwell systems.

In this paper, we study the linearized Vlasov-Maxwell systems by using a
framework of separable Hamiltonian systems, which was recently developed in
\cite{lin-zeng-EP} when studying the stability of nonrotating stars. Consider
a linear Hamiltonian PDEs of the separable form%
\begin{equation}
\partial_{t}\left(
\begin{array}
[c]{c}%
u\\
v
\end{array}
\right)  =\left(
\begin{array}
[c]{cc}%
0 & B\\
-B^{\prime} & 0
\end{array}
\right)  \left(
\begin{array}
[c]{cc}%
L & 0\\
0 & A
\end{array}
\right)  \left(
\begin{array}
[c]{c}%
u\\
v
\end{array}
\right)  =\mathbf{JL}\left(
\begin{array}
[c]{c}%
u\\
v
\end{array}
\right)  , \label{hamiltonian-separated}%
\end{equation}
where $u\in X,\ v\in Y$ and $X,Y$ are real Hilbert spaces. The triple $\left(
L,A,B\right)  $ is assumed to satisfy assumptions (G1)-(G3) in Section
\ref{section-abstract}, which roughly speaking require that $B:Y^{\ast}\supset
D(B)\rightarrow X$ \ is a densely defined closed operator, $L:X\rightarrow
X^{\ast}$ is bounded and self-dual with finitely many nonpositive modes, and
$A:Y\rightarrow Y^{\ast}$ is bounded, self-dual and positive. Then the number
of unstable modes of (\ref{hamiltonian-separated}) is shown to be equal to
$n^{-}\left(  L|_{\overline{R\left(  B\right)  }}\right)  $, which is the
number of negative directions of the quadratic form $\left\langle L\cdot
,\cdot\right\rangle $ restricted to the subspace $\overline{R\left(  B\right)
}\subset X$. Moreover, exponential trichotomy estimates are obtained for the
solution group $e^{tJL}$. See Theorem \ref{T:abstract} for the detailed
statements. By using a parity splitting of the distribution function, we are
able to rewrite the linearized $1\frac{1}{2}$D and $3$D Vlasov-Maxwell systems
in the separable Hamiltonian forms (\ref{hamiltonian-separated}) with the
constraint of the Poisson equation for the electric fields
((\ref{constraint-E}) for $1\frac{1}{2}$D and (\ref{eqn-maxwell-3}) for $3$D).
The assumption (\textbf{G1-3}) can be verified in an energy space $X$ and the
number $n^{-}\left(  L|_{\overline{R\left(  B\right)  }}\right)  $ is shown to
be exactly equal to $n^{-}\left(  \mathcal{L}^{0}\right)  $. Then by Theorem
\ref{T:abstract}, we recover the stability criterion obtained in
\cite{lin-strauss1} and \cite{lin-strauss2}. Moreover, we also obtain the
exponential trichotomy estimates for the linearized Vlasov-Maxwell systems.
These estimates will be useful for proving nonlinear instability or
constructing invariant (stable, unstable and center) manifolds near an
unstable steady state. The exponential trichotomy for the linearized
relativistic $1\frac{1}{2}$D Vlasov-Maxwell system can be shown
(\cite{lin-strauss2} \cite{guo-strauss}) by using the compact perturbation
($A$-smoothing) theory of semigroups, where the separation of characteristics
of the relativistic Vlasov equation and Maxwell system played a crucial role
in the proof. Such a separation is possible since the particle velocity in the
relativistic case is always less than the speed of light which is the
propagation speed of the Maxwell systems. However, for the nonrelativistic
Vlasov-Maxwell system such a separation of characteristics is no longer true
since the particle might travel faster than the speed of light, and as a
consequence the same arguments fail. By using the separable Hamiltonian
structures, the exponential trichotomy is obtained for both relativistic and
nonrelativistic Vlasov-Maxwell system. Moreover, we get more precise growth
estimates (i.e. at most quadratic growth) on the center space. In particular,
there is Liapunov stability on the center space when $\mathcal{L}^{0}$ has no kernel.

We make some comments to compare the Hamiltonian approach and the previous
approach. In \cite{lin-strauss1} \cite{lin-strauss2}, the instability and
stability criteria were obtained in very different ways. In the Hamiltonian
approach, both stability and instability information are obtained from the
computation of $n^{-}\left(  L|_{\overline{R\left(  B\right)  }}\right)  $.
Another difference lies in the treatment of the Poisson constraint. In the
Hamiltonian approach, the Poisson constraint is only imposed on the initial
data and it does not appear in the Hamiltonian formulation
(\ref{hamiltonian-separated}). Moreover, since the constraint is automatically
satisfied on the eigenspaces of nonzero eigenvalues, it does not affect the
counting of unstable modes. Thus, we can leave out the Poisson constraint
until stating the exponential trichotomy estimates for data satisfying this
constraint. We refer to Remark \ref{remark-constraint-1.5} for more details.
In \cite{lin-strauss1} \cite{lin-strauss2}, the Poisson equation is needed to
formulate a family of self-adjoint operators $\mathcal{A}^{\lambda\ }$on
electromagnetic potentials for the eigenvalue problem. But it requires some
careful choice of the gauge condition to make the Poisson equation to be
compatible with the current equation ((\ref{linearized-M1}) in $1\frac{1}{2}$D
and (\ref{eqn-maxwell-1}) in $3$D). The approach of (\cite{lin-strauss1}
\cite{lin-strauss2}) had been extended to Vlasov-Maxwell systems in a bounded
domain (\cite{nguen-strauss-arma} \cite{nguen-strauss-sima} \cite{zhang-sima}%
). It might still be possible to use the Hamiltonian formulation for models
with boundary conditions. The current Hamiltonian approach requires the
monotone dependence of steady distribution function on the particle energy. On
the other hand, the approach of (\cite{lin-strauss1} \cite{lin-strauss2}) can
be used to obtain sufficient instability conditions for non-monotonic steady
distribution function. See \cite{lin-01} \cite{guo-lin-stellar} for the
Vlasov-Poisson models, and \cite[Section 9]{lin-strauss2} \cite{ben-artz-jmp}
\cite{ben-artz-sima} \cite{ben-zrtz-survey} for the Vlasov-Maxwell models. It
would be very interesting to explore the Hamiltonian formulations for the
non-monotonic cases.

This paper is organized as follows. In Section 2, we state the results of
separable Hamiltonian systems to be used in later sections. In Sections 3, we
study the $1\frac{1}{2}$D Vlasov-Maxwell system. In Sections 4, we study the
$3$D relativistic Vlasov-Maxwell system with cylindrical symmetry.

\section{\label{section-abstract}Separable Linear Hamiltonian PDEs}

We briefly describe the results in \cite{lin-zeng-EP} about general separable
Hamiltonian PDEs (\ref{hamiltonian-separated}). The triple $\left(
L,A,B\right)  $ is assumed to satisfy assumptions:

\begin{enumerate}
\item[(\textbf{G1})] The operator $B:Y^{\ast}\supset D(B)\rightarrow X$ and
its dual operator $B^{\prime}:X^{\ast}\supset D(B^{\prime})\rightarrow Y\ $are
densely defined and closed (and thus $B^{\prime\prime}=B$).

\item[(\textbf{G2})] The operator $A:Y\rightarrow Y^{\ast}$ is bounded and
self-dual (i.e. $A^{\prime}=A$ and thus $\left\langle Au,v\right\rangle $ is a
bounded symmetric bilinear form on $Y$). Moreover, there exist $\delta>0$ such
that
\[
\langle Au,u\rangle\geq\delta\left\Vert u\right\Vert _{Y}^{2},\;\forall u\in
Y.
\]

\item[(\textbf{G3})] The operator $L:X\rightarrow X^{\ast}$ is bounded and
self-dual (i.e. $L^{\prime}=L$ \textit{etc.}) and there exists a decomposition
of $X$ into the direct sum of three closed subspaces
\begin{equation}
X=X_{-}\oplus\ker L\oplus X_{+},\ \dim\ker L<\infty,\ \ n^{-}(L)\triangleq\dim
X_{-}<\infty\label{decom-X}%
\end{equation}
satisfying

\begin{enumerate}
\item[(\textbf{G3.a})] $\left\langle Lu,u\right\rangle <0$ for all $u\in
X_{-}\backslash\{0\}$;

\item[(\textbf{G3.b})] there exists $\delta>0$ such that
\[
\left\langle Lu,u\right\rangle \geq\delta\left\Vert u\right\Vert ^{2}\ ,\text{
for any }u\in X_{+}.
\]

\end{enumerate}
\end{enumerate}

We note that the assumptions $\dim\ker L<\infty$ and $A>0\ $can be relaxed
(see \cite{lin-zeng-EP}). But these simplified assumptions are enough for the
applications to Vlasov-Maxwell systems studied in this paper.

\begin{theorem}
\label{T:abstract} \cite{lin-zeng-EP}Assume (\textbf{G1-3}) for
(\ref{hamiltonian-separated}). The operator $\mathbf{JL}$ generates a $C^{0}$
group $e^{t\mathbf{JL}}$ of bounded linear operators on $\mathbf{X}=X\times Y$
and there exists a decomposition%
\[
\mathbf{X}=E^{u}\oplus E^{c}\oplus E^{s},\quad
\]
of closed subspaces $E^{u,s,c}$ with the following properties:

i) $E^{c},E^{u},E^{s}$ are invariant under $e^{t\mathbf{JL}}$.

ii) $E^{u}\left(  E^{s}\right)  $ only consists of eigenvectors corresponding
to negative (positive) eigenvalues of $\mathbf{JL}$ and
\begin{equation}
\dim E^{u}=\dim E^{s}=n^{-}\left(  L|_{\overline{R\left(  B\right)  }}\right)
, \label{unstable-dimension-formula}%
\end{equation}
where $n^{-}\left(  L|_{\overline{R\left(  B\right)  }}\right)  $ denotes the
number of negative modes of
%$\left\langle L\cdot,\cdot\right\rangle $ restricted on $\overline{R\left(  B\right)  }$,
$\left\langle L\cdot,\cdot\right\rangle |_{\overline{R\left(  B\right)  }}$.
If $n^{-}\left(  L|_{\overline{R\left(  B\right)  }}\right)  >0$,
%$\BFJ\BFL|_{E^{u,s}}$ can be diagonalized with
%\[
%- \sigma(\BFJ\BFL|_{E^s}) = \sigma(\BFJ\BFL|_{E^u}) \subset \BFR^+
%\]
%and thus
then there exists $M>0$ such that
\begin{equation}
\left\vert e^{t\mathbf{JL}}|_{E^{s}}\right\vert \leq Me^{-\lambda_{u}%
t},\;t\geq0;\quad\left\vert e^{t\mathbf{JL}}|_{E^{u}}\right\vert \leq
Me^{\lambda_{u}t},\;t\leq0, \label{estimate-stable-unstable}%
\end{equation}
where $\lambda_{u}=\min\{\lambda\mid\lambda\in\sigma(\mathbf{JL}|_{E^{u}%
})\}>0$.

iii) The quadratic form $\left\langle \mathbf{L}\cdot,\cdot\right\rangle
%_{\mathbf{X}}
$ vanishes on $E^{u,s}$, i.e. $\langle\mathbf{L}\mathbf{u},\mathbf{u}%
\rangle=0$ for all $\mathbf{u}\in E^{u,s}$, but is non-degenerate on
$E^{u}\oplus E^{s}$, and
\begin{equation}
E^{c}=\left\{  \mathbf{u}\in\mathbf{X}\mid\left\langle \mathbf{\mathbf{L}%
u,v}\right\rangle =0,\ \forall\ \mathbf{v}\in E^{s}\oplus E^{u}\right\}  .
\label{defn-center space}%
\end{equation}
There exists $M>0$ such that
\begin{equation}
|e^{t\mathbf{J}\mathbf{L}}|_{E^{c}}|\leq M(1+t^{2}),\text{ for all }%
t\in\mathbf{R}. \label{estimate-center}%
\end{equation}

iv) Suppose $\left\langle L\cdot,\cdot\right\rangle $ is non-degenerate on
$\overline{R\left(  B\right)  }$, then $|e^{t\mathbf{JL}}|_{E^{c}}|\leq M$ for
some $M>0$. Namely, there is Lyapunov stability on the center space $E^{c}$.
\end{theorem}

\begin{remark}
Above theorem shows that the solutions of (\ref{hamiltonian-separated}) are
spectrally stable (i.e. nonexistence of exponentially growing solution) if and
only if $L|_{\overline{R\left(  B\right)  }}\geq0$. Moreover, $n^{-}\left(
L|_{\overline{R\left(  B\right)  }}\right)  $ gives the number of unstable
modes when $L|_{\overline{R\left(  B\right)  }}$ has a negative direction. The
exponential trichotomy estimates (\ref{estimate-stable-unstable}%
)-(\ref{estimate-center}) are important in the study of nonlinear dynamics
near an unstable steady state, such as the proof of nonlinear instability or
the construction of invariant (stable, unstable and center) manifolds. If the
spaces $E^{u,s}$ have higher regularity, then the exponential trichotomy can
be lifted to more regular spaces. We refer to Theorem 2.2 in
\cite{lin-zeng-hamiltonian} for more precise statements.
\end{remark}

\section{$1.5\ $D Vlasov-Maxwell systems}

In this section, we consider the stability of a class of equilibria of
$1\frac{1}{2}$D\ Vlasov-Maxwell systems by using the framework of separable
Hamiltonian systems. We largely follow the notations in (\cite{lin-strauss1}).
Here, we consider the classical (i.e. nonrelativistic) Vlasov-Maxwell system,
while in (\cite{lin-strauss1}) the relativistic Vlasov-Maxwell system was
studied. The stability criteria obtained in both cases are very similar.

The $1\frac{1}{2}$D Vlasov Maxwell system for electrons with a constant ion
background $n_{0}$ is
\[
\partial_{t}f+v_{1}\partial_{x}f-\left(  E_{1}+v_{2}B\right)  \partial_{v_{1}%
}f-\left(  E_{2}-v_{1}B\right)  \partial_{v_{2}}f=0
\]%
\[
\partial_{t}E_{1}=-j_{1}=\int v_{1}f\text{ }dv,\text{ }\partial_{t}%
B=-\partial_{x}E_{2}%
\]%
\[
\partial_{t}E_{2}+\partial_{x}B=-j_{2}=\int v_{2}f\text{ }dv
\]
with the constraint
\begin{subequations}
\[
\partial_{x}E_{1}=n_{0}-\int f\text{ }dv.
\]
We consider steady solutions of above system that are periodic in the variable
$x$ with a given period $P$. Consider the $P-$periodic equilibrium $f^{0}%
=\mu(e,p),$ $E_{1}^{0}=-\partial_{x}\phi^{0},E_{2}^{0}=0,B^{0}=\partial
_{x}\psi^{0},$ where the electromagnetic potentials $\left(  \phi^{0},\psi
^{0}\right)  $ satisfy the ODE system%
\end{subequations}
\[
\partial_{x}^{2}\phi^{0}=n_{0}-\int\mu(e,p)dv,\quad\partial_{x}^{2}\psi
^{0}=\int v_{2}\mu(e,p)dv
\]
with the electron energy and the \textquotedblleft angular momentum" defined
by
\begin{equation}
e=\frac{1}{2}\left\vert v\right\vert ^{2}-\phi^{0}(x),\quad p=v_{2}-\psi
^{0}(x). \label{defn-e-p}%
\end{equation}
We assume
\begin{equation}
\mu\geq0,\quad\mu\in C^{1},\ \ \mu_{e}\equiv\frac{\partial\mu}{\partial e}<0
\label{mu-assumption}%
\end{equation}
and, in order for $\int\left(  |\mu_{e}|+\left\vert \mu_{p}\right\vert
\right)  dv$ to be finite,
\begin{equation}
\left(  |\mu_{e}|+\left\vert \mu_{p}\right\vert \right)  \left(  e,p\right)
\leq c(1+\left\vert e\right\vert )^{-\frac{\alpha}{2}}\text{ for some }%
\alpha>2. \label{mu-assumption1}%
\end{equation}
The linearized Vlasov equation is
\begin{equation}
(\partial_{t}+D)f=\mu_{e}v_{1}E_{1}-\mu_{p}v_{1}B+(\mu_{e}v_{2}+\mu_{p})E_{2},
\label{linearized-V}%
\end{equation}
where $D$ is the transport operator associated with the steady fields, that
is,
\begin{align}
D  &  =v_{1}\partial_{x}-\left(  E_{1}^{0}+v_{2}B^{0}\right)  \partial_{v_{1}%
}+v_{1}B^{0}\partial_{v_{2}}\label{transport-operator}\\
&  =v_{1}\partial_{x}+\partial_{x}\phi^{0}\ \partial_{v_{1}}+\partial_{x}%
\psi^{0}\ (v_{1}\partial_{v_{2}}-v_{2}\partial_{v_{1}}).\nonumber
\end{align}
The linearized Maxwell equations become%
\begin{equation}
\partial_{t}E_{1}=\int v_{1}fdv, \label{linearized-M1}%
\end{equation}%
\begin{equation}
\ \partial_{t}E_{2}+\partial_{x}B=\int v_{2}fdv, \label{linearized-M2}%
\end{equation}%
\begin{equation}
\ \partial_{t}B+\partial_{x}E_{2}=0. \label{linearized-M3}%
\end{equation}
with the constraint
\begin{equation}
\partial_{x}E_{1}=-\int fdv. \label{constraint-E}%
\end{equation}
We consider the initial data satisfying the constraint $\int B\left(
0,x\right)  \ dx=0$. Then by (\ref{linearized-M3}), $\int B\left(  t,x\right)
\ dx=0$ for all $t\in\mathbf{R}$. Let $\psi\left(  t,x\right)  $ be the
magnetic potential function satisfying
\begin{equation}
\psi_{x}=B,\ \ \int_{0}^{P}\psi\left(  t,x\right)  dx=-\int_{0}^{t}\int%
_{0}^{P}E_{2}\left(  s,x\right)  dxdt. \label{defn-mag-potential}%
\end{equation}
Then by (\ref{linearized-M3}), $\psi_{t}=-E_{2}$. Below, we write the
linearized equations (\ref{linearized-V}) and (\ref{linearized-M1}%
)-(\ref{linearized-M3}) as a separable Hamiltonian system
(\ref{hamiltonian-separated}). We split $f$ into its even and odd parts in the
variable $v_{1}$:
\[
f=f_{ev}+f_{od},\quad\text{where }f_{ev}(x,v_{1},v_{2})=\tfrac{1}%
{2}\{f(x,v_{1},v_{2})+f(x,-v_{1},v_{2})\}.
\]
and define $g_{ev}=$ $f_{ev}+\mu_{p}\psi$. The operator $D$ takes even
functions into odd ones, and vice versa. So from (\ref{linearized-V}), we
have
\begin{align}
\partial_{t}f_{od}  &  =-Df_{ev}+(E_{1}+v_{2}B)\partial_{v_{1}}f^{0}%
-v_{1}B\partial_{v_{2}}f^{0}\label{eqn-f-od}\\
&  =-Df_{ev}+\mu_{e}v_{1}E_{1}-\mu_{p}v_{1}\partial_{x}\psi=-Dg_{ev}+\mu
_{e}v_{1}E_{1},\nonumber
\end{align}
and
\[
\partial_{t}f_{ev}+Df_{od}=E_{2}\partial_{v_{2}}f^{0}=\mu_{e}v_{2}E_{2}%
-\mu_{p}\partial_{t}\psi,
\]
which yields
\begin{equation}
\partial_{t}g_{ev}=-Df_{od}+\mu_{e}v_{2}E_{2}. \label{eqn-g-ev}%
\end{equation}
The Maxwell equations (\ref{linearized-M1})-(\ref{linearized-M3}) become
\begin{equation}
\partial_{t}E_{1}=\int v_{1}f_{od}dv, \label{eqn-E1}%
\end{equation}%
\begin{equation}
\partial_{t}E_{2}=-\partial_{xx}\psi-\int\mu_{p}v_{2}\psi+\int v_{2}g_{ev}dv,
\label{eqn-E2}%
\end{equation}%
\begin{equation}
\partial_{t}\psi=-E_{2}. \label{eqn-psi}%
\end{equation}
Define
\[
X_{od}=\left\{  f\in L_{\frac{1}{|\mu_{e}|}}^{2}\ |\ f(x,-v_{1},v_{2}%
)=-f(x,v_{1},v_{2})\text{ }\right\}
\]
and%

\[
X_{ev}=\left\{  f\in L_{\frac{1}{|\mu_{e}|}}^{2}\ |\ f(x,-v_{1},v_{2}%
)=f(x,v_{1},v_{2})\text{ }\right\}  .
\]
Let $L_{P}^{2},H_{P}^{1}$ be the $x-$periodic functions in $L^{2}$ and $H^{1}%
$, and define $X=X_{ev}\times L_{P}^{2}\times H_{P}^{1}$. Define the operators
$L:X\rightarrow X^{\ast}$ by
\begin{equation}
L\left(
\begin{array}
[c]{c}%
g_{ev}\\
E_{1}\\
\psi
\end{array}
\right)  =\left(
\begin{array}
[c]{ccc}%
-\frac{1}{\mu_{e}} & 0 & 0\\
0 & I & 0\\
0 & 0 & L_{0}%
\end{array}
\right)  \left(
\begin{array}
[c]{c}%
g_{ev}\\
E_{1}\\
\psi
\end{array}
\right)  ,\ \ \label{operator-L}%
\end{equation}
where $L_{0}=-\frac{d^{2}}{dx^{2}}-\int\mu_{p}v_{2}dv$. Let $Y=X_{od}\times
L_{P}^{2}$ and define the operator $A:Y\rightarrow Y^{\ast}$ by
\begin{equation}
A=\left(
\begin{array}
[c]{cc}%
-\frac{1}{\mu_{e}} & 0\\
0 & I
\end{array}
\right)  . \label{operator-A}%
\end{equation}
Note that $A:Y\rightarrow Y^{\ast}$ is an isometry. Define $B:Y^{\ast}\supset
D(B)\rightarrow X$ by
\begin{equation}
B=\left(
\begin{array}
[c]{cc}%
\mu_{e}D & \mu_{e}v_{2}\\
-\int\mu_{e}v_{1}\cdot dv & 0\\
0 & -I
\end{array}
\right)  \label{operator B}%
\end{equation}
and the corresponding dual operator $B^{\prime}:X^{\ast}\supset D(B^{\prime
})\rightarrow Y$ is
\[
B^{\prime}=\left(
\begin{array}
[c]{ccc}%
-\mu_{e}D & -\mu_{e}v_{1} & 0\\
\int\mu_{e}v_{2}\cdot dv & 0 & -I
\end{array}
\right)  .
\]
Let
\[
u=\left(
\begin{array}
[c]{c}%
g_{ev}\\
E_{1}\\
\psi
\end{array}
\right)  \in X,\ \ v=\left(
\begin{array}
[c]{c}%
f_{od}\\
E_{2}%
\end{array}
\right)  \in Y.
\]
Then the linearized $1\frac{1}{2}$D Vlasov-Maxwell system (\ref{eqn-f-od}%
)-(\ref{eqn-psi}) can be written as a separable Hamiltonian form
(\ref{hamiltonian-separated}) with $\left\langle L,A,B\right\rangle $ defined
in (\ref{operator-L})-(\ref{operator B}). Now we check that the triple
$\left\langle L,A,B\right\rangle $ satisfies assumptions (\textbf{G1-3}) in
Section \ref{section-abstract}. Assumptions (\textbf{G1-2}) are obvious. To
verify (\textbf{G3}), we note that for any $\left(  g_{ev},E_{1},\psi\right)
\in X$,
\begin{align}
&  \left\langle L\left(
\begin{array}
[c]{c}%
g_{ev}\\
E_{1}\\
\psi
\end{array}
\right)  ,\left(
\begin{array}
[c]{c}%
g_{ev}\\
E_{1}\\
\psi
\end{array}
\right)  \right\rangle \label{quadratic_L}\\
&  =\iint\frac{1}{|\mu_{e}|}\left\vert g_{ev}\right\vert ^{2}dvdx+\int%
\left\vert E_{1}\right\vert ^{2}dx+\int\left\vert \psi^{\prime}\right\vert
^{2}dx-\int\int\mu_{p}v_{2}\left\vert \psi\right\vert ^{2}dxdv.\nonumber
\end{align}
Then assumption (\textbf{G3}) follows since the operator $L_{0}=-\frac{d^{2}%
}{dx^{2}}-\int\mu_{p}v_{2}dv$ has finite-dimensional negative and zero
eigenspaces. To apply Theorem \ref{T:abstract} to study the solutions of
(\ref{eqn-f-od})-(\ref{eqn-psi}), we need to compute $n^{-}\left(
L|_{\overline{R\left(  B\right)  }}\right)  $. First, we introduce some
notations as in (\cite{lin-strauss1}). Define the following operators,
$\mathcal{A}_{1}^{0},\mathcal{A}_{2}^{0},\mathcal{L}^{0}\ $act from $H_{P}%
^{2}$ to $L_{P}^{2}$ and $\mathcal{B}^{0},\left(  \mathcal{B}^{0}\right)
^{\ast}$ act from $L_{P}^{2}$ to $L_{P}^{2}$
\[
\mathcal{A}_{1}^{0}h=-\partial_{x}^{2}h-\left(  \int\mu_{e}dv\right)
h+\int\mu_{e}\ \mathcal{P}h\ dv,
\]%
\[
\mathcal{A}_{2}^{0}h=-\partial_{x}^{2}h-\left(  \int v_{2}\mu_{p}dv\right)
h-\int\mu_{e}v_{2}\mathcal{P}(\hat{v}_{2}h)\ dv,
\]%
\begin{align*}
\mathcal{B}^{0}h  &  =\left(  \int\mu_{p}dv\right)  h+\int\mu_{e}%
\ \mathcal{P}(v_{2}h)\ dv\\
&  =-\int\mu_{e}\left(  I-\mathcal{P}\right)  \left(  v_{2}h\right)  dv,\\
\left(  \mathcal{B}^{0}\right)  ^{\ast}h  &  =\left(  \int\mu_{p}dv\right)
h+\int v_{2}\mu_{e}\mathcal{P}(h)\ dv
\end{align*}
and
\begin{equation}
\mathcal{L}^{0}=(\mathcal{B}^{0})^{\ast}(\mathcal{A}_{1}^{0})^{-1}%
\mathcal{B}^{0}+\mathcal{A}_{2}^{0}, \label{defn-cal-L0-1.5}%
\end{equation}
where $\mathcal{P}$ is the projection operator of $L_{\left\vert \mu
_{e}\right\vert }^{2}$ onto $\ker D$. Then we have

\begin{lemma}
\label{lemma-dimension-1.5d}
\[
n^{-}\left(  L|_{\overline{R\left(  B\right)  }}\right)  =n^{-}\left(
\mathcal{L}^{0}\right)  ,\ \dim\ker L|_{\overline{R\left(  B\right)  }}%
=\dim\ker\mathcal{L}^{0}.
\]

\end{lemma}

\begin{proof}
First, for any $0\neq u=\left(  g_{ev},E_{1},\psi\right)  \in X$ with
$\left\langle Lu,u\right\rangle \leq0$, it is easy to see from
(\ref{quadratic_L}) that $\psi\neq0$. For any $u=\left(  g_{ev},E_{1}%
,\psi\right)  \in R\left(  B\right)  =R\left(  BA\right)  $, let $u=BAv$ where
$v=\left(  f_{od},E_{2}\right)  \in Y$. Then
\[
g_{ev}=-Df_{od}+\mu_{e}v_{2}E_{2},\ \ E_{1}=\int v_{1}f_{od}dv,\ \psi
=-E_{2}\text{. }%
\]
Thus
\begin{align*}
\left\langle Lu,u\right\rangle  &  =\iint\frac{1}{|\mu_{e}|}\left\vert
Df_{od}-\mu_{e}v_{2}E_{2}\right\vert ^{2}dvdx+\int\left\vert \partial_{x}%
E_{2}\right\vert ^{2}dx\\
&  +\int\left\vert \int v_{1}f_{od}dv\right\vert ^{2}dx-\iint\mu_{p}%
v_{2}\left\vert E_{2}\right\vert ^{2}dvdx\\
&  :=W\left(  f_{od},E_{2}\right)  .
\end{align*}
It was shown in (\cite[P. 751-752]{lin-strauss1}) that $W\left(  f_{od}%
,E_{2}\right)  \geq\left(  \mathcal{L}^{0}E_{2},E_{2}\right)  $. Therefore,
$\left\langle Lu,u\right\rangle \geq\left(  \mathcal{L}^{0}\psi,\psi\right)  $
for any $u\in R\left(  B\right)  $, and also for any $u\in$ $\overline
{R\left(  B\right)  }$ by the density argument. Thus, $n^{\leq0}\left(
L|_{\overline{R\left(  B\right)  }}\right)  \leq n^{\leq0}\left(
\mathcal{L}^{0}\right)  $, where $n^{\leq0}\left(  L|_{\overline{R\left(
B\right)  }}\right)  $ and $n^{\leq0}\left(  \mathcal{L}^{0}\right)  $ denote
the maximal dimensions of subspaces where the quadratic forms $\langle
L\cdot,\cdot\rangle|_{\overline{R\left(  B\right)  }}$ and $(\mathcal{L}%
^{0}\cdot,\cdot)$ are nonpositive.

Next we show that $n^{\leq0}\left(  L|_{\overline{R\left(  B\right)  }%
}\right)  \geq n^{\leq0}\left(  \mathcal{L}^{0}\right)  $, which then implies
that $n^{\leq0}\left(  L|_{\overline{R\left(  B\right)  }}\right)  =n^{\leq
0}\left(  \mathcal{L}^{0}\right)  $. For any $\psi\in H_{P}^{1}$, define
\begin{equation}
\phi^{\psi}=-\left(  \mathcal{A}_{1}^{0}\right)  ^{-1}\mathcal{B}^{0}%
\psi,\ \ f^{\psi}=\mu_{p}\psi-\mu_{e}\phi^{\psi}+\mu_{e}\mathcal{P}(v_{2}%
\psi+\phi^{\psi}). \label{defn-E-psi}%
\end{equation}
Then by the definition of $\phi^{\psi}$
\[
\frac{d^{2}}{dx^{2}}\phi^{\psi}\left(  x\right)  =\int f^{\psi}dv.
\]
Let
\begin{equation}
E_{1}^{\psi}=\frac{d}{dx}\phi^{\psi}\left(  x\right)  ,\ \ g_{ev}^{\psi
}=-f^{\psi}+\mu_{p}\psi. \label{defn-g-psi}%
\end{equation}
We show that $u^{\psi}=\left(  g_{ev}^{\psi},\ E_{1}^{\psi},\psi\right)  \in$
$\overline{R\left(  B\right)  }$. Indeed, since $g_{ev}^{\psi}\in X_{ev}$ and
\[
g_{ev}^{\psi}+\mu_{e}v_{2}\psi=\mu_{e}\left(  I-\mathcal{P}\right)  (v_{2}%
\psi+\phi^{\psi})\in\overline{R\left(  D\right)  },
\]
there exists a sequence $\left\{  h_{od}^{n}\right\}  \in X_{od}\cap
Dom\left(  D\right)  $ such that
\[
\left\Vert -Dh_{od}^{n}-\left(  g_{ev}^{\psi}+\mu_{e}v_{2}\psi\right)
\right\Vert _{L_{\frac{1}{|\mu_{e}|}}^{2}}\rightarrow0\text{, when
}n\rightarrow\infty\text{. }%
\]
We can choose $h_{od}^{n}$ such that
\begin{equation}
\int\int v_{1}h_{od}^{n}dxdv=0. \label{condition-zero-mean-f-od}%
\end{equation}
To show this, we claim that there exists an odd (in $v_{1}$) function $\chi
\in\ker D$ such that $\int\int v_{1}\chi dxdv\neq0$. Therefore, we can adjust
$h_{od}^{n}$ by $c\chi$ to ensure (\ref{condition-zero-mean-f-od}). Indeed, a
function $\chi\in\ker D$ if and only if it takes constant values on each
particle trajectory $\left(  X\left(  t\right)  ,V_{1}\left(  t\right)
,V_{1}\left(  t\right)  \right)  \ $in the steady electromagnetic fields
\[
\left(  E_{1}^{0},E_{2}^{0},B^{0}\right)  =\left(  -\partial_{x}\phi
^{0},0,\partial_{x}\psi^{0}\right)  ,
\]
that is,
\[
\dot{X}\left(  t\right)  =V_{1},\ \dot{V}_{1}=-\left(  E_{1}^{0}\left(
X\right)  +V_{2}B^{0}\left(  X\right)  \right)  ,\ \dot{V}_{2}=V_{1}%
B^{0}\left(  X\right)  .
\]
In particular, $\chi$ can take opposite constants on two untrapped particle
trajectories with the same particle energy $e$ and momentum $p$ (defined in
\ref{defn-e-p})) satisfying
\[
e>\max\left[  \left(  p+\psi^{0}\right)  ^{2}-\phi^{0}\left(  x\right)
\right]
\]
but with different sign of $v_{1}$. By choosing $\chi\in\ker D$ to be zero on
the trapped region, take nonnegative values on the untrapped trajectory with
positive $v_{1}$ and opposite values on the other untrapped trajectory with
negative $v_{1}$, we can ensure that $\int\int\mu_{e}v_{1}\chi\ dxdv<0$. We
note that this also implies that $v_{1}\notin\overline{R\left(  D\right)
}=\left(  \ker D\right)  ^{\perp}$.

Let
\[
E_{1}^{n}=\int v_{1}h_{od}^{n}dv,\ g_{ev}^{n}=-Dh_{od}^{n}-\mu_{e}v_{2}\psi,
\]
then
\[
u_{n}=\left(  g_{ev}^{n},E_{1}^{n},\psi\right)  =BA\left(
\begin{array}
[c]{c}%
h_{od}^{n}\\
-\psi
\end{array}
\right)  \in R\left(  B\right)  .
\]
Moreover, the property (\ref{condition-zero-mean-f-od}) implies that
$E_{1}^{n}=\frac{d}{dx}\phi^{n}$, where
\[
\frac{d^{2}}{dx^{2}}\phi^{n}=-\int\left(  g_{ev}^{n}+\mu_{e}v_{2}\psi\right)
dv=\frac{d}{dx}\int v_{1}h_{od}^{n}dv.
\]
Since
\[
\frac{d^{2}}{dx^{2}}\phi^{\psi}=\int f^{\psi}dv=-\int\left(  g_{ev}^{\psi}%
-\mu_{p}\psi\right)  dv=-\int\left(  g_{ev}^{\psi}+\mu_{e}v_{2}\psi\right)
dv
\]
and $\left\Vert g_{ev}^{n}-g_{ev}^{\psi}\right\Vert _{L_{\frac{1}{|\mu_{e}|}%
}^{2}}\rightarrow0$, thus $\left\Vert E_{1}^{n}-E_{1}^{\psi}\right\Vert
_{L^{2}}\rightarrow0$ when $n\rightarrow\infty$. This shows that $\left\Vert
u^{\psi}-u^{n}\right\Vert _{X}\rightarrow0$ and $u^{\psi}\in\overline{R\left(
B\right)  }$. As shown in the proof of Lemma 2.8 in \cite{lin-strauss1}, we
have
\[
\left(  \mathcal{L}^{0}E_{2},E_{2}\right)  =L\left(  u^{\psi},u^{\psi}\right)
.
\]
Thus $n^{\leq0}\left(  \mathcal{L}^{0}\right)  \leq n^{\leq0}\left(
L|_{\overline{R\left(  B\right)  }}\right)  $ which implies $n^{\leq0}\left(
\mathcal{L}^{0}\right)  =n^{\leq0}\left(  L|_{\overline{R\left(  B\right)  }%
}\right)  $. To show that $n^{-}\left(  L|_{\overline{R\left(  B\right)  }%
}\right)  =n^{-}\left(  \mathcal{L}^{0}\right)  $, it remains to show that
\begin{equation}
\dim\ker L|_{\overline{R\left(  B\right)  }}=\dim\ker\mathcal{L}^{0}.
\label{dim-equal}%
\end{equation}
We note that $u\in\ker L|_{\overline{R\left(  B\right)  }}$ is equivalent to
$u=\left(  g_{ev},E_{1},\psi\right)  \in\overline{R\left(  B\right)  }\cap
\ker\left(  B^{\prime}L\right)  $. So
\begin{equation}
Dg_{ev}-\mu_{e}v_{1}E_{1}=0 \label{eqn-ker-g}%
\end{equation}
and
\begin{equation}
L_{0}\psi+\int v_{2}g_{ev}dv=0. \label{eqn-ker-psi}%
\end{equation}
Since $u\in\overline{R\left(  B\right)  }$, we have
\begin{equation}
\mathcal{P}\left(  g_{ev}+\mu_{e}v_{2}\psi\right)  =0,\ \ \label{range-1}%
\end{equation}%
\begin{equation}
\frac{d}{dx}E_{1}=-\int\left(  g_{ev}+\mu_{e}v_{2}\psi\right)  dv.
\label{range-2}%
\end{equation}
Let $\phi$ be such that
\begin{equation}
\phi_{xx}=-\int\left(  g_{ev}+\mu_{e}v_{2}\psi\right)  dv. \label{eqn-Poisson}%
\end{equation}
Then $E_{1}=\phi_{x}+k$ where $k=\frac{1}{P}\int_{0}^{P}E_{1}dx$. By
(\ref{eqn-ker-g}), we have $D\left(  g_{ev}-\mu_{e}\phi\right)  =k\mu_{e}%
v_{1}$ which implies that $k=0$ since $\mu_{e}v_{1}\notin\overline{R\left(
D\right)  }$. Thus $D\left(  g_{ev}-\mu_{e}\phi\right)  =0$, that is, $\left(
I-\mathcal{P}\right)  \left(  g_{ev}-\mu_{e}\phi\right)  =0$. Combining with
(\ref{range-1}), we get
\begin{equation}
g_{ev}=\mu_{e}\phi-\mu_{e}\mathcal{P}\left(  v_{2}\psi\right)  -\mu
_{e}\mathcal{P}\phi. \label{g-formula}%
\end{equation}
Plugging above into (\ref{eqn-Poisson}), we get $\mathcal{A}_{1}^{0}%
\phi=\mathcal{B}^{0}\psi$ and $\phi=\left(  \mathcal{A}_{1}^{0}\right)
^{-1}\mathcal{B}^{0}\psi$. Then by combining with (\ref{eqn-ker-psi}) and
(\ref{g-formula}), it yields $\mathcal{L}^{0}\psi=0$. On the other hand, if
$\mathcal{L}^{0}\psi=0$, define $E_{1}^{\psi}$ and $g_{ev}^{\psi}$ as in
(\ref{defn-E-psi}) and (\ref{defn-g-psi}). Then $\left(  g_{ev}^{\psi}%
,E_{1}^{\psi},\psi\right)  \in\overline{R\left(  B\right)  }$. By reversing
the above computation, it can be checked that (\ref{eqn-ker-g}) and
(\ref{eqn-ker-psi}) are satisfied. This shows that $\left(  g_{ev}^{\psi
},E_{1}^{\psi},\psi\right)  \in\ker\left(  L|_{\overline{R\left(  B\right)  }%
}\right)  $ . Thus $\ker\left(  L|_{\overline{R\left(  B\right)  }}\right)  $
and $\ker\mathcal{L}^{0}$ have the same dimension. This finishes the proof of
the lemma.
\end{proof}

\begin{remark}
\label{remark-constraint-1.5}We make some comments on the constraint
(\ref{constraint-E}) which becomes
\begin{equation}
\partial_{x}E_{1}=-\int\left(  g_{ev}-\mu_{p}\psi\right)  .
\label{constraint-E-ev}%
\end{equation}
This constraint is preserved by the system (\ref{eqn-f-od})-(\ref{eqn-psi}) in
the sense that
\[
\partial_{t}\left(  \partial_{x}E_{1}+\int\left(  g_{ev}-\mu_{p}\psi\right)
\right)  =0.
\]
In particular, this implies that for any nonzero eigenvalue $\lambda\ $of
(\ref{eqn-f-od})-(\ref{eqn-psi}), the constraint (\ref{constraint-E-ev}) is
satisfied on the corresponding eigenspace. Therefore, the same dimension
formula (\ref{unstable-dimension-formula}) is true under the constraint
(\ref{constraint-E-ev}). The exponential trichotomy estimates
(\ref{estimate-stable-unstable})-(\ref{estimate-center}) remain the same by
restricting to initial data satisfying the constraint (\ref{constraint-E-ev}).
The same remark applies to the constraint $\int_{0}^{P}B\left(  x,t\right)
dx=0$.
\end{remark}

We can apply Theorem \ref{T:abstract} to the linearized system (\ref{eqn-f-od}%
)-(\ref{eqn-psi}) with initial data satisfying the constraints $\int_{0}%
^{P}B\left(  x,0\right)  dx=0$ and (\ref{constraint-E-ev}). To be more
convenient for potential applications to nonlinear problems, we state the
results without the even and odd splitting of $f$. Let $\psi\left(
x,t\right)  $ be the magnetic potential defined in (\ref{defn-mag-potential})
and define $g=f+\mu_{p}\psi$. Then $g$ satisfies the equation
\begin{equation}
g_{t}=-Dg+\mu_{e}v_{1}E_{1}+\mu_{e}v_{2}E_{2}\label{eqn-g-1.5}%
\end{equation}
by (\ref{eqn-f-od}) and (\ref{eqn-g-ev}). The Maxwell system becomes
\[
\partial_{t}E_{1}=\int v_{1}g\ dv,
\]%
\[
\partial_{t}E_{2}=-\partial_{xx}\psi-\int\mu_{p}v_{2}\psi+\int v_{2}gdv,
\]%
\[
\partial_{t}\psi=-E_{2},
\]
with the constraint
\begin{equation}
\partial_{x}E_{1}=-\int\left(  g-\mu_{p}\psi\right)  .\label{constraint-E-g}%
\end{equation}

\begin{theorem}
\label{T:1.5D VM}Consider the above equivalent linearized Vlasov-Maxwell
systems for$\ \left(  g,E_{1},E_{2},\psi\right)  $ in the space
\[
\mathbf{Z}=L_{\frac{1}{|\mu_{e}|}}^{2}\times L_{P}^{2}\times L_{P}^{2}\times
H_{P}^{1},
\]
with initial data satisfying the constraint (\ref{constraint-E-g}). Then

i) The solution mapping is strongly continuous in the space $\mathbf{Z}$ and
there exists a decomposition%
\[
\mathbf{Z}=E^{u}\oplus E^{c}\oplus E^{s},\quad
\]
of closed subspaces $E^{u,s,c}$ with the following properties:

i) $E^{c},E^{u},E^{s}$ are invariant under the linearized system.

ii) $E^{u}\left(  E^{s}\right)  $ only consists of eigenvectors corresponding
to negative (positive) eigenvalues of the linearized system and
\[
\dim E^{u}=\dim E^{s}=n^{-}\left(  \mathcal{L}^{0}\right)  ,
\]
where $\mathcal{L}^{0}$ is defined in (\ref{defn-cal-L0-1.5}). In particular,
$\mathcal{L}^{0}\geq0$ implies spectral stability.

iii) The exponential trichotomy is true in the space $Z$ in the sense of
(\ref{estimate-stable-unstable})-(\ref{estimate-center}). Moreover, if
$\ker\mathcal{L}=\left\{  0\right\}  $, then Liapunov stability is true under
the norm $\left\Vert {}\right\Vert _{Z}\ $on the center space $E^{c}$.
\end{theorem}

By assuming $\int\frac{\left\vert \mu_{p}\right\vert ^{2}}{\left\vert \mu
_{e}\right\vert }dv<\infty$, above Theorem implies the exponential trichotomy
for the linearized VM system (\ref{linearized-V}), (\ref{linearized-M1}%
)-(\ref{linearized-M3}) for $\left(  f,E_{1},E_{2},B\right)  \ $in the norm
\[
\left\Vert f\right\Vert _{L_{\frac{1}{|\mu_{e}|}}^{2}}+\left\Vert
E_{1}\right\Vert _{L^{2}}+\left\Vert E_{2}\right\Vert _{L^{2}}+\left\Vert
B\right\Vert _{L^{2}}.
\]

\section{$3$D Vlasov-Maxwell systems}

The case of 3D Vlasov-Maxwell is rather similar to the 1.5D case. So we will
be more sketchy and only give details when there are significant differences.

As in \cite{lin-strauss1} and \cite{lin-strauss2}, we consider the 3D
relativistic Vlasov-Maxwell system (RVM) for a non-neutral electron plasma
with external fields
\[
\partial_{t}f+\hat{v}\cdot\nabla_{x}f-(\mathbf{E}+\mathbf{E}^{ext}+\hat
{v}\times\left(  \mathbf{B}+\mathbf{B}^{ext}\right)  )\cdot\nabla_{v}f=0
\]%
\begin{equation}
\partial_{t}\mathbf{E}-\nabla\times\mathbf{B}=\int\hat{v}f\text{
}dv=-\mathbf{j} \label{eqn-maxwell-1}%
\end{equation}%
\begin{equation}
\partial_{t}\mathbf{B}+\nabla\times\mathbf{E}=0,\ \ \nabla\cdot\mathbf{B}=0
\label{eqn-maxwell-2}%
\end{equation}%
\begin{equation}
\nabla\cdot\mathbf{E}=-\int f\text{ }dv=\rho,\quad\label{eqn-maxwell-3}%
\end{equation}
where $x\in\mathbb{R}^{3},v\in\mathbb{R}^{3}$. Denote $\left(  r,\theta
,z\right)  $ to be the cylindrical coordinates. The equilibrium distribution
function with cylindrical symmetry is assumed to have the form $f^{0}%
=\mu\left(  e,p\right)  ,$ where
\[
e=\sqrt{1+\left\vert v\right\vert ^{2}}-\phi^{0}\left(  r,z\right)
-\phi^{ext}\left(  r,z\right)  ,
\]%
\[
p=r\left(  v_{\theta}-A_{\theta}^{0}\left(  r,z\right)  -A_{\theta}%
^{ext}\left(  r,z\right)  \right)  ,
\]
are particle energy and momentum, and $\left(  \phi^{0}\left(  r,z\right)
,A_{\theta}^{0}\left(  r,z\right)  \right)  $ and $\left(  \phi^{ext}\left(
r,z\right)  ,A_{\theta}^{ext}\left(  r,z\right)  \right)  $ are self-generated
and external electromagnetic potentials. The steady electromagnetic fields are
given by
\[
\mathbf{E}^{0}=-\partial_{r}\phi^{0}\vec{e}_{r}-\partial_{z}\phi^{0}\vec
{e}_{z},\text{ \ }\mathbf{B}^{0}=-\partial_{z}A_{\theta}^{0}\vec{e}_{r}%
+\frac{1}{r}\partial_{r}\left(  rA_{\theta}^{0}\right)  \vec{e}_{z}.
\]
The steady potentials $\left(  A_{\theta}^{0},\phi^{0}\right)  $ satisfy the
elliptic system%
\begin{equation}
\Delta\phi^{0}=\partial_{zz}\phi^{0}+\partial_{rr}\phi^{0}+\frac{1}{r}%
\partial_{r}\phi^{0}=\int\mu dv, \label{eqn-equi-3d-ele}%
\end{equation}%
\begin{equation}
\left(  \Delta-\frac{1}{r^{2}}\right)  A_{\theta}^{0}=\partial_{zz}A_{\theta
}^{0}+\partial_{rr}A_{\theta}^{0}+\frac{1}{r}\partial_{r}A_{\theta}^{0}%
-\frac{1}{r^{2}}A_{\theta}^{0}=\int\hat{v}_{\theta}\mu dv.
\label{eqn-equi-3d-mag}%
\end{equation}
By choosing $\phi^{ext}$, $A_{\theta}^{ext}$ and $\mu$ properly, steady
solutions satisfying (\ref{eqn-equi-3d-ele})-(\ref{eqn-equi-3d-mag}) were
constructed in \cite{lin-strauss1} with a compact support $S$ for $f^{0}$ in
the $(x,v)$ space and $f^{0},E^{0},B^{0}$ to be differentiable in the whole
space. We assume that $\mu_{e}<0$ on the support $\{\mu>0\}$. The linearized
VM systems are
\begin{equation}
\partial_{t}f+Df-(\mathbf{E}+\hat{v}\times\mathbf{B})\cdot\nabla_{v}f^{0}=0,
\label{linearized-V-3D}%
\end{equation}
coupled with the Maxwell systems (\ref{eqn-maxwell-1})-(\ref{eqn-maxwell-3}).
Here,
\[
D=\hat{v}\cdot\nabla_{x}-\left(  \mathbf{E}^{0}+\mathbf{E}^{ext}+\hat{v}%
\times\left(  \mathbf{B}^{0}+\mathbf{B}^{ext}\right)  \right)  \cdot\nabla_{v}%
\]
is the transport operator with the steady electromagnetic fields. We consider
axi-symmetric perturbations and decompose such $f$ as $f=f_{od}+f_{ev}$ where
$f_{od}\ \left(  f_{ev}\right)  $ is odd (even) in $\left(  v_{r}%
,v_{z}\right)  .$ Then the linearized Vlasov equation (\ref{linearized-V-3D})
can be written as (see \cite{lin-strauss1})
\begin{equation}
\partial_{t}f_{od}+Df_{ev}=\mu_{e}\left(  \hat{v}_{r}E_{r}+\hat{v}_{z}%
E_{z}\right)  -\mu_{p}r\left(  \hat{v}_{r}B_{z}-\hat{v}_{z}B_{r}\right)  ,
\label{eqn-f-od-1}%
\end{equation}
and%
\begin{equation}
\partial_{t}f_{ev}+Df_{od}=\mu_{e}\hat{v}_{\theta}E_{\theta}+\mu_{p}%
rE_{\theta}. \label{eqn-f-ev}%
\end{equation}
Introduce the magnetic potential function $A_{\theta}$ such that
$B_{r}=-\partial_{z}A_{\theta},\ B_{z}=\frac{1}{r}\partial_{r}\left(
rA_{\theta}\right)  $ and
\begin{equation}
\partial_{t}A_{\theta}=-E_{\theta}. \label{eqn-A-theta}%
\end{equation}
Define $g_{ev}=f_{ev}+r\mu_{p}A_{\theta}\,$and note that $r\left(  \hat{v}%
_{r}B_{z}-\hat{v}_{z}B_{r}\right)  =D\left(  rA_{\theta}\right)  $, then we
can get from (\ref{eqn-f-od-1})-(\ref{eqn-f-ev})
\begin{equation}
\partial_{t}f_{od}=-Dg_{ev}+\mu_{e}\left(  \hat{v}_{r}E_{r}+\hat{v}_{z}%
E_{z}\right)  \label{eqn-f-od-3d}%
\end{equation}%
\begin{equation}
\partial_{t}g_{ev}=-Df_{od}+\mu_{e}\hat{v}_{\theta}E_{\theta}.
\label{eqn-g-ev-3d}%
\end{equation}
The Maxwell system (\ref{eqn-maxwell-1})-(\ref{eqn-maxwell-3}) is reduced to
\begin{equation}
\partial_{t}E_{r}=-\partial_{z}B_{\theta}+\int\hat{v}_{r}f_{od}\ dv,\ \partial
_{t}E_{z}=\frac{1}{r}\partial_{r}\left(  rB_{\theta}\right)  +\int\hat{v}%
_{z}f_{od}\ dv, \label{eqn-Er-Ez}%
\end{equation}%
\begin{equation}
\partial_{t}B_{\theta}=-\partial_{z}E_{r}+\partial_{r}E_{z},
\label{eqn-B-theta}%
\end{equation}%
\begin{equation}
\partial_{t}E_{\theta}=\partial_{z}B_{r}-\partial_{r}B_{z}+\int\hat{v}%
_{\theta}f_{ev}dv=L_{0}A_{\theta}+\int\hat{v}_{\theta}g_{ev}dv,
\label{eqn-E-theta}%
\end{equation}%
\[
L_{0}=-\partial_{zz}-\partial_{rr}-\frac{1}{r}\partial_{r}+\frac{1}{r^{2}%
}-\int\hat{v}_{\theta}\mu_{p}dv\text{ }r,
\]
with the constraint
\begin{equation}
\nabla\cdot\mathbf{E}=\frac{1}{r}\partial_{r}\left(  rE_{r}\right)
+\partial_{z}E_{z}=-\int g_{ev}dv+\int r\mu_{p}dv\ A_{\theta}.
\label{constraint-3d}%
\end{equation}
Define
\[
X_{od}=\left\{  f\in L_{\frac{1}{|\mu_{e}|}}^{2}\left(  \mathbf{R}^{3}%
\times\mathbf{R}^{3}\right)  \ |\ f(r,z,-v_{r},v_{\theta},-v_{z}%
)=-f(r,z,v_{r},v_{\theta},v_{z})\text{ }\right\}
\]
and%

\[
X_{ev}=\left\{  f\in L_{\frac{1}{|\mu_{e}|}}^{2}\left(  \mathbf{R}^{3}%
\times\mathbf{R}^{3}\right)  \ |\ f(r,z,-v_{r},v_{\theta},-v_{z}%
)=f(r,z,v_{r},v_{\theta},v_{z})\text{ }\right\}  .
\]
Let $V^{1}$ to be the space of cylindrically symmetric functions $h\left(
r,z\right)  $ such that
\[
\left\Vert h\right\Vert _{V^{1}}=\left(  \int\left(  \left\vert \frac{1}%
{r}\partial_{r}\left(  rh\right)  \right\vert ^{2}+\left\vert \partial
_{z}h\right\vert ^{2}\right)  dx\right)  ^{\frac{1}{2}}=\left\Vert
\nabla\left(  he^{i\theta}\right)  \right\Vert _{L^{2}\left(  \mathbf{R}%
^{3}\right)  }<\infty,
\]
and $L_{s}^{2}$ be the space of cylindrically symmetric functions in
$L^{2}\left(  \mathbf{R}^{3}\right)  $. Let $X=X_{ev}\times\left(  L_{s}%
^{2}\right)  ^{2}\times V^{1}$ and $Y=X_{od}\times\left(  L_{s}^{2}\right)
^{2}$. Define the isometry operator $A:Y\rightarrow Y^{\ast}$ by
\begin{equation}
A=\left(
\begin{array}
[c]{ccc}%
-\frac{1}{\mu_{e}} & 0 & \\
0 & I & \\
&  & I
\end{array}
\right)  \label{defn-A-3d}%
\end{equation}
and $L:X\rightarrow X^{\ast}$ by
\begin{equation}
L=\left(
\begin{array}
[c]{cccc}%
-\frac{1}{\mu_{e}} &  &  & \\
& I &  & \\
&  & I & \\
&  &  & L_{0}%
\end{array}
\right)  . \label{defn-L-3d}%
\end{equation}
Define $B:Y^{\ast}\supset D(B)\rightarrow X$ \ by%
\begin{equation}
B=\left(
\begin{array}
[c]{ccc}%
\mu_{e}D & \mu_{e}\hat{v}_{\theta} & 0\\
-\int\hat{v}_{r}\cdot dv & 0 & -\partial_{z}\\
-\int\hat{v}_{z}\cdot dv & 0 & \frac{1}{r}\partial_{r}\left(  r\cdot\right) \\
0 & -I & 0
\end{array}
\right)  \label{defn-B-3d}%
\end{equation}
and the dual operator $B^{\prime}:X^{\ast}\supset D(B^{\prime})\rightarrow Y$
is%
\[
B^{\prime}=\left(
\begin{array}
[c]{cccc}%
-\mu_{e}D & -\mu_{e}\hat{v}_{r} & -\mu_{e}\hat{v}_{z} & 0\\
\int\mu_{e}\hat{v}_{\theta}\cdot dv & 0 & 0 & -I\\
0 & \partial_{z} & -\partial_{r} & 0
\end{array}
\right)  .
\]
Let $u=\left(  g_{ev},E_{r},E_{z},A_{\theta}\right)  \in X$ and $v=\left(
f_{od},E_{\theta},B_{\theta}\right)  \in Y$, then the linearized $3$D
relativistic Vlasov-Maxwell system (\ref{eqn-A-theta})-(\ref{eqn-E-theta}) can
be written as a separable Hamiltonian form (\ref{hamiltonian-separated}) with
$\left\langle L,A,B\right\rangle $ defined in (\ref{defn-A-3d}%
)-(\ref{defn-B-3d}). We check that the triple $\left\langle L,A,B\right\rangle
$ satisfies assumptions (\textbf{G1-3}) in Section \ref{section-abstract}. We
note that for any $u=\left(  g_{ev},E_{r},E_{z},A_{\theta}\right)  \in X$ ,
\[
\left\langle Lu,u\right\rangle =\iint\frac{1}{|\mu_{e}|}\left\vert
g_{ev}\right\vert ^{2}dvdx+\int\left\vert E_{r}\right\vert ^{2}dx+\int%
\left\vert E_{z}\right\vert ^{2}dx+\left\langle L_{0}A_{\theta},A_{\theta
}\right\rangle ,
\]
where
\[
\left\langle L_{0}A_{\theta},A_{\theta}\right\rangle =\int\left(  \left\vert
\partial_{z}A_{\theta}\right\vert ^{2}+\left\vert \frac{1}{r}\partial
_{r}\left(  rA_{\theta}\right)  \right\vert ^{2}\right)  dx-\int\int r\hat
{v}_{\theta}\mu_{p}\left\vert A_{\theta}\right\vert ^{2}dxdv.
\]
Note that since $f^{0}=\mu\left(  e,p\right)  $ has compact support in $x,v$,
we have
\[
\left\vert \int r\hat{v}_{\theta}\mu_{p}\left\vert A_{\theta}\right\vert
^{2}dxdv\right\vert \lesssim\left\Vert A_{\theta}\right\Vert _{L^{6}}%
^{2}\lesssim\left\Vert \nabla\left(  A_{\theta}e^{i\theta}\right)  \right\Vert
_{L^{2}\left(  \mathbf{R}^{3}\right)  }^{2}=\left\Vert A_{\theta}\right\Vert
_{V^{1}}^{2}.
\]
Moreover, by Lemma 3.1 of \cite{lin-strauss1} and its proof, $\sigma
_{\text{ess}}\left(  L_{0}\right)  =[0,\infty)$ and $L_{0}$ is a relative
compact perturbation of
\[
\left(  -\Delta\right)  _{mag}:=-\partial_{zz}-\partial_{rr}-\frac{1}%
{r}\partial_{r}+\frac{1}{r^{2}}.
\]
Thus there exists a finite co-dimensional subspace $V\subset V^{1}$ such that
\[
\left\langle L_{0}A_{\theta},A_{\theta}\right\rangle \geq c_{0}\left\Vert
A_{\theta}\right\Vert _{V^{1}}^{2},\ \ \forall\ \ A_{\theta}\in V,
\]
for some constant $c_{0}>0$. This proves assumption (G3) and assumptions
(G1)-(G2) are obvious.

Now we compute $n^{-}\left(  L|_{\overline{R\left(  B\right)  }}\right)  $. As
in \cite{lin-strauss1}, we define the following operators acting on the
cylindrically symmetric functions $h=h\left(  r,z\right)  \in L^{2}%
(\mathbb{R}^{3})\ $by
\[
\mathcal{A}_{1}^{0}h=-\partial_{zz}h-\partial_{rr}h-\frac{1}{r}\partial
_{r}h-\int\mu_{e}dvh+\int\mu_{e}\mathcal{P}\left(  h\right)  dv,
\]%
\[
\mathcal{A}_{2}^{0}h=-\partial_{zz}h-\partial_{rr}h-\frac{1}{r}\partial
_{r}h+\frac{1}{r^{2}}h-\int\hat{v}_{\theta}\mu_{p}dv\text{ }rh-\int\hat
{v}_{\theta}\mu_{e}\mathcal{P}\left(  \hat{v}_{\theta}h\right)  dv,
\]%
\[
\mathcal{B}^{0}h=\int\mu_{e}\mathcal{P}\left(  \hat{v}_{\theta}h\right)
dv-\int\hat{v}_{\theta}\mu_{e}dv\text{ }h,
\]
and
\begin{equation}
\mathcal{L}^{0}=\left(  \mathcal{B}^{0}\right)  ^{\ast}\left(  \mathcal{A}%
_{1}^{0}\right)  ^{-1}\mathcal{B}^{0}+\mathcal{A}_{2}^{0},
\label{defn-cal-L0-3d}%
\end{equation}
where $\mathcal{P}$ is the projection operator of $L_{\left\vert \mu
_{e}\right\vert }^{2}$ onto $\ker D$. The properties of these operators are
studied in Lemma 3.1 of \cite{lin-strauss1}. As in the $1\frac{1}{2}$D case,
we have

\begin{lemma}
\label{lemma-dimension-3d}For $L,B$ defined in (\ref{defn-L-3d}) and
(\ref{defn-B-3d}), we have
\[
n^{-}\left(  L|_{\overline{R\left(  B\right)  }}\right)  =n^{-}\left(
\mathcal{L}^{0}\right)  ,\ \dim\ker L|_{\overline{R\left(  B\right)  }}%
=\dim\ker\mathcal{L}^{0}.
\]

\end{lemma}

\begin{proof}
The proof is similar to that of Lemma \ref{lemma-dimension-1.5d}. For any
$u=\left(  g_{ev},E_{r},E_{z},A_{\theta}\right)  \in R\left(  B\right)
=R\left(  BA\right)  $, let $u=BAv$ where $v=$ $\left(  f_{od},E_{\theta
},B_{\theta}\right)  \in Y$. Then%
\begin{align*}
g_{ev}  &  =-Df_{od}+\mu_{e}\hat{v}_{\theta}E_{\theta},\ A_{\theta}%
=-E_{\theta},\\
E_{r}  &  =-\partial_{z}B_{\theta}+\int\hat{v}_{r}f_{od}\ dv,\ E_{z}=\frac
{1}{r}\partial_{r}\left(  rB_{\theta}\right)  +\int\hat{v}_{z}f_{od}\ dv,
\end{align*}
and
\begin{align*}
&  \left\langle Lu,u\right\rangle \\
&  =\iint\frac{1}{\left\vert \mu_{e}\right\vert }\left\vert Df_{od}-\mu
_{e}\hat{v}_{\theta}E_{\theta}\right\vert ^{2}dxdv-\int r\hat{v}_{\theta}%
\mu_{p}\left\vert E_{\theta}\right\vert ^{2}dxdv+\int\left(  \left\vert
\frac{1}{r}\partial_{r}\left(  rE_{\theta}\right)  \right\vert ^{2}+\left\vert
\partial_{z}E_{\theta}\right\vert ^{2}\right)  dx\\
&  +\int\left(  \left\vert -\partial_{z}B_{\theta}+\int\hat{v}_{r}%
f_{od}\ dv\right\vert ^{2}+\left\vert \frac{1}{r}\partial_{r}\left(
rB_{\theta}\right)  +\int\hat{v}_{z}f_{od}\ dv\right\vert ^{2}\right)  dx\\
&  :=W\left(  f_{od},E_{\theta},B_{\theta}\right)  .
\end{align*}
It was shown in \cite{lin-strauss1} that $W\left(  f_{od},E_{\theta}%
,B_{\theta}\right)  \geq\left(  \mathcal{L}^{0}E_{\theta},E_{\theta}\right)
$. Thus $\left\langle Lu,u\right\rangle \geq\left(  \mathcal{L}^{0}A_{\theta
},A_{\theta}\right)  $ for any $u=\left(  g_{ev},E_{r},E_{z},A_{\theta
}\right)  \in R\left(  B\right)  $, which yields $n^{\leq0}\left(
L|_{\overline{R\left(  B\right)  }}\right)  \leq n^{\leq0}\left(
\mathcal{L}^{0}\right)  $ as in the $1\frac{1}{2}$D case.

Next, we show $n^{\leq0}\left(  L|_{\overline{R\left(  B\right)  }}\right)
\geq n^{\leq0}\left(  \mathcal{L}^{0}\right)  $. For any $A_{\theta}\in V^{1}%
$, we define
\[
\phi^{A_{\theta}}=-\left(  \mathcal{A}_{1}^{0}\right)  ^{-1}\mathcal{B}%
^{0}A_{\theta},\ \ f^{A_{\theta}}=r\mu_{p}A_{\theta}-\mu_{e}\phi^{A_{\theta}%
}+\mu_{e}\mathcal{P}(\hat{v}_{\theta}A_{\theta}+\phi^{A_{\theta}}).
\]
By the definition of $\phi^{A_{\theta}}$, we have $\Delta\phi^{A_{\theta}%
}=\int f^{A_{\theta}}dv$. Define
\begin{equation}
E_{r}^{A_{\theta}}=\partial_{r}\phi^{A_{\theta}},\ E_{z}^{A_{\theta}}%
=\partial_{z}\phi^{A_{\theta}},\ \ g_{ev}^{A_{\theta}}=-f^{A_{\theta}}%
+r\mu_{p}A_{\theta}. \label{formula-u-A-th}%
\end{equation}
Then $u^{A_{\theta}}=\left(  g_{ev}^{A_{\theta}},E_{r}^{A_{\theta}}%
,E_{z}^{A_{\theta}},A_{\theta}\right)  \in\overline{R\left(  B\right)  }$. We
skip the proof since it is similar to the $1\frac{1}{2}$D case. We only point
out that the following observation is used. Let $h\in Dom\left(  D\right)
\cap$ $L_{\frac{1}{|\mu_{e}|}}^{2}$, if $\Delta\phi=\int Dhdv$ and $\left(
E_{r},0,E_{z}\right)  =\nabla\phi$, then there exists $B_{\theta}\in L_{s}%
^{2}$ such that
\[
E_{r}=-\partial_{z}B_{\theta}+\int\hat{v}_{r}h\ dv,\ E_{z}=\frac{1}{r}%
\partial_{r}\left(  rB_{\theta}\right)  +\int\hat{v}_{z}h\ dv,
\]
which is due to
\[
\frac{1}{r}\partial_{r}\left(  r\left(  E_{r}-\int\hat{v}_{r}h\ dv\right)
\right)  +\partial_{z}\left(  E_{z}-\int\hat{v}_{z}h\ dv\right)  =\Delta
\phi-\int Dhdv=0.
\]
It is easy to check that $\left(  \mathcal{L}^{0}A_{\theta},A_{\theta}\right)
=L\left(  u^{A_{\theta}},u^{A_{\theta}}\right)  $. This shows that $n^{\leq
0}\left(  L|_{\overline{R\left(  B\right)  }}\right)  \geq n^{\leq0}\left(
\mathcal{L}^{0}\right)  $ and consequently $n^{\leq0}\left(  L|_{\overline
{R\left(  B\right)  }}\right)  =n^{\leq0}\left(  \mathcal{L}^{0}\right)  $.

It remains to prove $\dim\ker L|_{\overline{R\left(  B\right)  }}=\dim
\ker\mathcal{L}^{0}$. If $u=\left(  g_{ev},E_{r},E_{z},A_{\theta}\right)
\in\ker L|_{\overline{R\left(  B\right)  }}$, then $u\in\overline{R\left(
B\right)  }\cap\ker\left(  B^{\prime}L\right)  $. Thus
\begin{equation}
Dg_{ev}-\mu_{e}\hat{v}_{r}E_{r}-\mu_{e}\hat{v}_{z}E_{z}=0,
\label{eqn-ker-g-3d}%
\end{equation}%
\begin{equation}
L_{0}A_{\theta}+\int\hat{v}_{\theta}g_{ev}dv=0, \label{eqn-ker-A-th-3d}%
\end{equation}%
\begin{equation}
\partial_{z}E_{r}-\partial_{r}E_{z}=0. \label{eqn-ker-E-th-3d}%
\end{equation}
By (\ref{eqn-ker-E-th-3d}), there exists a potential function $\phi\left(
r,z\right)  $ such that $E_{r}=\partial_{r}\phi$ and $E_{z}=\partial_{z}\phi$.
Since $u\in\overline{R\left(  B\right)  },$ it follows that
\begin{equation}
\Delta\phi=\nabla\cdot\left(  E_{r},0,E_{z}\right)  =-\int\left(  g_{ev}%
+\mu_{e}\hat{v}_{\theta}A_{\theta}\right)  dv \label{eqn-phi-ker-3d}%
\end{equation}
By (\ref{eqn-ker-g-3d}), $D\left(  g_{ev}-\mu_{e}\phi\right)  =0$ which
implies $\left(  I-\mathcal{P}\right)  \left(  g_{ev}-\mu_{e}\phi\right)  =0$.
Since $u\in\overline{R\left(  B\right)  },\ \mathcal{P}\left(  g_{ev}+\mu
_{e}\hat{v}_{\theta}A_{\theta}\right)  =0$. Thus
\begin{equation}
g_{ev}=\mu_{e}\phi-\mu_{e}\mathcal{P}\left(  \hat{v}_{\theta}A_{\theta
}\right)  +\mu_{e}\mathcal{P}\phi. \label{formula-g-ev-ker-3d}%
\end{equation}
Combining (\ref{eqn-ker-A-th-3d}), (\ref{eqn-phi-ker-3d}) and
(\ref{formula-g-ev-ker-3d}), we get $\mathcal{L}^{0}A_{\theta}=0$. On the
other hand, if $\mathcal{L}^{0}A_{\theta}=0$, we define $u^{A_{\theta}%
}=\left(  g_{ev}^{A_{\theta}},E_{r}^{A_{\theta}},E_{z}^{A_{\theta}},A_{\theta
}\right)  \in\overline{R\left(  B\right)  }$ by (\ref{formula-u-A-th}). Then
reversing the above computation, we have $u^{A_{\theta}}\in\ker\left(
L|_{\overline{R\left(  B\right)  }}\right)  $. This shows that $\dim\ker
L|_{\overline{R\left(  B\right)  }}=\dim\ker\mathcal{L}^{0}$.
\end{proof}

Let $g=f+r\mu_{p}A_{\theta}$, which satisfies
\[
g_{t}=-Dg+\mu_{e}\left(  \hat{v}_{r}E_{r}+\hat{v}_{z}E_{z}+\hat{v}_{\theta
}E_{\theta}\right)  .
\]
Then we can study the equivalent linearized Vlasov-Maxwell system for $\left(
g,A_{\theta},B_{\theta},E_{\theta},E_{r},E_{z}\right)  $, where $\left(
A_{\theta},E_{r},E_{z},E_{\theta},B_{\theta}\right)  $ satisfy
\[
\partial_{t}A_{\theta}=-E_{\theta},\ \ \partial_{t}B_{\theta}=-\partial
_{z}E_{r}+\partial_{r}E_{z},\ \ \partial_{t}E_{\theta}=L_{0}A_{\theta}%
+\int\hat{v}_{\theta}g\ dv
\]%
\[
\partial_{t}E_{r}=-\partial_{z}B_{\theta}+\int\hat{v}_{r}g\ dv,\ \partial
_{t}E_{z}=\frac{1}{r}\partial_{r}\left(  rB_{\theta}\right)  +\int\hat{v}%
_{z}g\ dv,
\]
with the constraint
\begin{equation}
\frac{1}{r}\partial_{r}\left(  rE_{r}\right)  +\partial_{z}E_{z}=-\int
gdv+\int r\mu_{p}dv\ A_{\theta}. \label{constraint-3d-g}%
\end{equation}
As in the $1\frac{1}{2}$D case (Remark \ref{remark-constraint-1.5}), the
constraint (\ref{constraint-3d-g}) is automatically satisfied on the
eigenspaces of nonzero eigenvalues.

\begin{theorem}
Consider the linearized relativistic Vlasov-Maxwell system for $\left(
g,A_{\theta},B_{\theta},E_{\theta},E_{r},E_{z}\right)  $, with axi-symmetric
initial data in the space
\[
Z=L_{\frac{1}{|\mu_{e}|}}^{2}\left(  \mathbf{R}^{3}\times\mathbf{R}%
^{3}\right)  \times V^{1}\times\left(  L_{s}^{2}\left(  \mathbf{R}^{3}\right)
\right)  ^{4}%
\]
satisfying the constraint (\ref{constraint-3d-g}).Then

i) The solution mapping is strongly continuous in the space $\mathbf{Z}$ and
there exists a decomposition%
\[
\mathbf{Z}=E^{u}\oplus E^{c}\oplus E^{s},\quad
\]
of closed subspaces $E^{u,s,c}$ with the following properties:

i) $E^{c},E^{u},E^{s}$ are invariant under the linearized RVM system.

ii) $E^{u}\left(  E^{s}\right)  $ only consists of eigenvectors corresponding
to negative (positive) eigenvalues of the linearized system and
\[
\dim E^{u}=\dim E^{s}=n^{-}\left(  \mathcal{L}^{0}\right)  ,
\]
where $\mathcal{L}^{0}$ is defined in (\ref{defn-cal-L0-3d}). In particular,
$\mathcal{L}^{0}\geq0$ implies spectral stability.

iii) The exponential trichotomy is true in the space $Z$ in the sense of
(\ref{estimate-stable-unstable})-(\ref{estimate-center}). Moreover, if
$\ker\mathcal{L}=\left\{  0\right\}  $, then Liapunov stability is true under
the norm $\left\Vert {}\right\Vert _{Z}\ $on the center space $E^{c}$.
\end{theorem}

By assuming $\int\frac{\left\vert \mu_{p}\right\vert ^{2}}{\left\vert \mu
_{e}\right\vert }dv<\infty$, above Theorem implies the exponential trichotomy
for the original linearized RVM system (\ref{linearized-V}),
(\ref{linearized-M1})-(\ref{linearized-M3}) for $\left(  f,\mathbf{E}%
,\mathbf{B}\right)  \ $in the norm $\left\Vert f\right\Vert _{L_{\frac{1}%
{|\mu_{e}|}}^{2}}+\left\Vert \mathbf{E}\right\Vert _{L^{2}}+\left\Vert
\mathbf{B}\right\Vert _{L^{2}}$, where $\mathbf{E=}\left(  E_{r},E_{\theta
},E_{z}\right)  ,\mathbf{B=}\left(  B_{r},B_{\theta},B_{z}\right)  .$

\begin{center}
{\Large Acknowledgement}
\end{center}

This paper is dedicated to the memory of Robert Glassey. This work is
supported partly by the NSF grant DMS-2007457.

\end{document}